%% file: sing_index.tex
\documentclass{article}
\usepackage{config_en_art}

\title{On local real algebraic geometry and applications to kinematics}
\date{\today}
\author{Marc Diesse\\ Faculty of Mechanics and Electronic, Hochschule Heilbronn }

\begin{document}
\maketitle
\input{abstract}

\input{introduction}
\input{preliminaries}

\input{local_real}
\input{normalization}
\input{four_bar}
\input{higher_dim}

\bibliography{literatur}{}
\bibliographystyle{plain}
\end{document}

%% file: abstract.tex
\begin{abstract}
We address the question of identifying non-smooth points in $\algrv(I)$ the real part of an affine algebraic variety. A simple algebraic criterion will be formulated and proven. As an application
we can answer several questions asked in \cite{piippo:planar} about the configuration spaces of planar linkages and frameworks, respectively.
\end{abstract}

%% file: introduction.tex
\section{Introduction}
For any zero set $X = \algrv(I)$ of an ideal $I = (g_1, \ldots, g_k) \leq \R[\bx]$, $\bx = (x_1, \ldots, x_n)$, there is the question of identifying points where $X$ is not locally a submanifold of $\R^n$. 


The standard approach to this problem is to look for points $p \in X$, where the rank of the jacobian of $(g_1, \ldots, g_k)$ drops below the height of $I$, which is the codimension of $\algcv(I)$. 
Unfortunately this is in general not enough to imply, that $X$ is \textbf{not} locally a submanifold.
Obviously problems arise, if $I$ is not radical or equidimensional (cf. Ex.~\ref{ex:introduction} (ii),(iii)) and techniques to handle those problems are well known (although not computationally feasible in some cases), but there are more intricate difficulties for real algebraic sets, where the localization of the reduced coordinate ring is not regular and $X = \algrv(I)$ is still a smooth submanifold of $\R^n$ at this point (cf. Ex.~\ref{ex:introduction}~(vi)). 


The following examples show different kinds of behavior of real algebraic sets at points, where the jacobian drops rank.

\begin{exs}\label{ex:introduction} In all examples we set $\frm := \langle \bx \rangle \leq \R[\bx]$, $A=\R[\bx]/I$.
\begin{itemize}
	\item[(i)]	
	The simple node $I = \langle y^2 - x^2 - x^3 \rangle \leq \R[x,y]$, shows the expected behavior. $A_{\frm}$ is not regular and $X = \algrv(I)$ is not locally a manifold at the origin.

	\item[(ii)]
	Let $I = \langle x^2,x y \rangle \leq \R[x,y]$. Then $X = \algrv(I)$ is just the $y$-axis, which is locally a manifold at the origin, although $A_{\frm}$ is not regular. The problem here is clearly, that $I$ is not a radical ideal, i.e. $A_{\mathrm{red}} = A/\sqrt{(0)}$ localized at $\frm$ is regular. In theory $\sqrt{I}$ is algorithmically computable with Gr\"obner base methods (\texttt{radical(I)} calculates the radical in \texttt{Singular} for example). Unfortunately the computation is unfeasible in many cases. But we will see, that we can avoid the computation of the radical for many systems of polynomials, which come from engineering problems.

	\item[(iii)]
	Let $I = \langle (z-1)xy, z(z-1) \rangle \leq \R[x,y,z]$. Then $\algrv(I)$ is the union of the $x$-axis, the $y$-axis and the plane given by $z=1$. Obviously $X$ is locally not a submanifold at the origin, but the rank of the jacobian at the origin equals $\height I = 1$. Note that $I$ is radical, but not equidimensional and $A_{\frm}$ is not regular. In this case we need to calculate an equidimensional decomposition before applying the jacobian criterion. This is possible in general again with Gr\"obner base methods (\texttt{primdecGTZ(I)} calculates a primary decomposition in \texttt{Singular}) but as hard as the computation of the radical. We will see a well known criterion to decide when $I$ is already equidimensional, which is useful for many problems in kinematics.

	\item[(iv)]
	Let $I = \langle x^2 + y^2 \rangle \leq \R[x,y,z]$. Then $X = \algrv(I)$ is the $z$-axis, which is a submanifold of $\R^3$, although the rank of the jacobian drops at any point of $X$ and $I$ is radical and equidimensional. This difficulty only appears in real geometry, since $X_{\C} = \algv(I)$ is \textbf{not} locally a complex manifold at any point of the $z$-axis.
	The problem is that $I \neq \algi_\R(X) = \{f \in \R[x,y] \mid f|_{X} \equiv 0 \} = \langle x,y \rangle$, since clearly $(\R[x,y,z]/\langle x,y \rangle)_\frm$ is a regular local ring. 
	
	There are algorithms to compute the real radical $\sqrt[r]{I} = \algi_\Q(X)$ from $I \leq \Q[\bx]$ (e.g. \texttt{realrad(I)} computes the real radical over $\Q$ in \texttt{Singular}) but this computation is harder than that of the normal radical. Also we have in general $\sqrt[r]{I} \cdot \R[\bx] \ne \sqrt[r]{I \cdot \R[\bx]} = \algi_\R(X)$ (see example (v)), in contrast to the usual radical.
	If this is the case not much can be gained by computing $\algi_{\Q} (X)$. We will present a very useful criterion by T. Y. Lam \cite{lam:real_alg} to check for an ideal $I$ whether $\algi_\R(X) = I$.

	\item[(v)]
	Let $I = \langle x^3 - 5y^3 \rangle \leq \Q[x,y]$ and $X = \algrv(I)$, which is just the line given by $x = \sqrt[3]{5}\,y$. The jacobian drops rank at the origin but $X$ is an analytic submanifold of $\R^2$. Note that $I_\Q(X) = I$, but $I_\R(X) = \langle x - \sqrt[3]{5}\,y \rangle$.

	\item[(vi)]
	This example motivated this paper. Let $I = \langle y^3 + 2\,x^2\,y - x^4 \rangle \leq \R[x,y]$ and $X = \algrv(I)$. We will see, $A_{\frm}$ is not regular and even $\algi_\R(X) = I$, but $X = \algrv(I)$ is the analytic submanifold of $\R^2$ shown in figure~\ref{fig:curve_2xy2}. We notice, that $I = \algi_\R(X)$ and $A_\frm$ not regular does not imply, that $X$ is "nonsmooth" at the origin.
	
	The reason here is that some analytic branches are not visible in the real picture. We will carefully investigate this case by analyzing the completion of the local ring at this point.

	\item[(vii)]
	Let $I = \langle y^3 - x^{10} \rangle \leq \R[\bx]$ and $X = \algrv(I)$. Here $A_\frm$ is not regular and $\algi(X) = I$ again. But in this case $X$ is not locally an analytic submanifold at the origin although the real picture looks very "smooth", which is because $X$ is a $C^3$-submanifold (but not $C^4$). 
	
	It is well known, that any real algebraic set which is (locally) $C^{\infty}$ is also $C^{\omega}$, so any "nonanalytic" point is at the most "finitely differentiable". This example emphasizes the need for an algebraic criterion to algebraically discern between the singularities seen in the last two examples because the real picture can be very deceiving. Criteria to identify points which are not locally topological submanifolds are beyond the scope of this article, although we will see, that we can rule out this case in a lot of situations.	  
\end{itemize}
\end{exs}
In this paper we want to show strategies to effectively deal with all the problems seen in the examples when analyzing a singular point of a real algebraic set. This is of great importance in the theory of linkages \cite{muller:sing_conf} \cite{muller:local_geometry}, when studying local kinematic properties, since the configuration space of a linkage will usually be given as a real algebraic set. For a further discussion we refer the reader to the final sections~\ref{sec:four_bar} and \ref{sec:higher_dim}, where we investigate the configuration space of a class of planar linkages with the developed techniques. We will be able to address all questions raised in \cite{piippo:planar}.

The rest of the paper is structured as follows: In section 2 and 3 we review some well known facts from commutative algebra, real algebra and differential geometry, which will enable us to make precise the notion of manifold point and deal with examples (i)-(v). We will also put a focus on base extensions of affine algebras, which comes in very handy if one needs to extend results gained by calculations in polynomial rings over $\Q$ to polynomial rings over $\R$.  

In section 4 we will build the theoretical foundation for local analysis of real algebraic sets. Central to the exposition is theorem~\ref{thm:mani_reg} which gives an algebraic condition for manifold points and will show together with Risler's analytic Nullstellensatz, that this is an intrinsic property.

Section 5 deals with the problem, that the extensions of a prime ideal of $\R[\bx]$ to the ring of formal power series $\R[[\bx]]$ will not be prime in general and symbolic calculations in $\R[[\bx]]$ are not possible effectively. Instead we will investigate the integral closure of the local ring to divide the analytic branches. The main result Theorem~\ref{thm:normal} goes back to Zariski~\cite{zariski:comm_alg} and was extended by Ruiz~\cite{ruiz:power_series} to a complete description of the normalization of $\CF$, where $\CF$ is the local ring $\R[[\bx]]/\left(I \cdot \R[[\bx]] \right)$.

Finally in section~6 we formulate and prove Theorem~\ref{thm:curves}, which decides the case completely for real algebraic curves. This extends results of \cite{maria:plane_curves}. 

For further reading regarding local properties of real algebraic sets the following authors - without whom (among many others) this paper would not have been possible - are recommended: H. Whitney~\cite{whitney:local_prop}, \cite{whitney:tangents} for his work on (tangents of) analytic varieties. T. Lam~\cite{lam:real_alg} for his real algebra introduction. The book of J. Ruiz~\cite{ruiz:power_series} covering all basic theory of power series rings. G. Efroymson~\cite{efroymson:local_reality} for the fundamental work on the realness of local ring completions, R. Risler for the (analytic) real Nullstellensatz~\cite{risler:real_nullstellensatz},\cite{ruiz:power_series} and D. O'Shea, L. Charles \cite{shea:limits_tangent} for their work on geometric Nash fibers, limits of tangent spaces and real tangent cones.
\begin{figure}\label{fig:curve_2xy2}
\centering
\begin{tikzpicture}
\begin{axis}[xmin=-1,xmax=0.5*2,ymin=0, ymax=1, xtick={-0.8,0.8}, ytick={-0.8,0.8}, unit vector ratio*=1 1 1, axis lines=center, axis on top, xlabel={$x$}, ylabel={$y$}]
\addplot[very thick, blue] file {data_2xy2};
\end{axis}
\end{tikzpicture}%
\caption{$\algrv(y^3 + 2x^2y - x^4)$}%
\end{figure}

%% file: preliminaries.tex
\section{Algebraic Preliminaries}\label{sec:alg_prelim}
In this section let $\K$ be a field with $\Q \subset \K \subset \R$, $f_1, \ldots, f_n$ a set of polynomials in $\K[\bx]$, where $\bx = (x_1, \ldots, x_n)$, and $I = \langle f_1, \ldots, f_n \rangle \leq \K[\bx]$ the ideal generated by the $f_i$. We set $A = \K[\bx]/I$ and consider two sets associated to $A$:
\begin{align*}
X_\C & := \{\, x \in \C^n \mid f(x) = 0, \text{for all $f \in I$} \,\} = \algv(I), \\
X & := \{\, x \in \R^n \mid f(x) = 0, \text{for all $f \in I$} \,\} = \algrv(I).
\end{align*}
Sometimes we will call $X$ the \defm{real picture} of $X_\C$. 
Since we can usually only perform symbolic computations over the rational numbers we need to investigate base changes of $A$. For any extension field $\K \subset \K'$ and any ideal $J \leq A$ we set
\begin{align*}
I_{\K'} & := \K' \otimes_\K I = I \cdot \K'[\bx], \quad  A_{\K'} := \K' \otimes_\K A = \K'[\bx]/ I_{\K'} \\
J_{\K'} & := \K' \otimes_\K J = ( \hat{J} \cdot \K'[\bx]) / I_{\K'}, \quad \text{where $\hat{J} \leq \K[\bx]$ with $\hat{J}/I = J$}. 
\end{align*}
If $\K' = \C$, we call $A_\C$, $I_{\C}$ or $J_{\C}$ the complexification of $A$, $I$ or $J$ respectively. Finally for any $p = (p_1, \ldots, p_n) \in \C^n$ we define the maximal ideal
\[
\frm_p = \langle x_1 - p_1, \ldots, x_n - p_n \rangle \subset \C[\bx].
\]

\begin{definition}
The \defm{singular locus} of $A$ is the set of all prime ideals $\frp \in \spec A$ such that $A_{\frp}$ is not regular. A point $p \in X_{\C}$ is called a \defm{singularity} of $X_{\C}$ if $((A_{\C})_{\mathrm{red}})_{\frm_p}$ is not regular, i.e. $\frm_p$ is in the singular locus of $(A_{\C})_{\mathrm{red}}$.
\end{definition}

\begin{remark}
$(A_{\C})_{\mathrm{red}}$ denotes the reduction $A_{\C}/\sqrt{(0)}$ without nilpotents.
The stacking of subscripts in $((A_{\C})_{\mathrm{red}})_{\frm_p}$ is admittedly horrible but we will see in Proposition~\ref{prop:base_change}, that there is a certain kind of freedom in the choice of the coefficient field. So we can get rid of the complexification and/or the reduction if $I$ is radical and/or $\frm_p \leq A$. 
\end{remark}

\subsection{Base Change}
We review some facts from commutative algebra regarding extensions of the coefficient field.

\begin{prop}[Base Change]\label{prop:base_change} Let $\K'$ be any field extension of $\K$, where $\mathrm{char}(\K) = 0$. Then
	\begin{itemize}
		\item[(i)] $I_{\K'} \cap \K[\bx] = I$.
		\item[(ii)] $\height I_{\K'} = \height I$, $\dim A_{\K'} = \dim A$.
		\item[(iii)] $\sqrt{I_{\K'}} = \sqrt{I} \, \K'[\bx]$.
		\item[(iv)] Let $\frp \leq A$ prime. Then $A_{\frp}$ is regular iff $(A_{\K'})_{\frP}$ is regular for one and then all associated primes $\frP$ of $\frp_{\K'}$.
	\end{itemize}
\end{prop}

\begin{remark}
Since we require $\mathrm{char}(\K) = 0$, $\K$ is a perfect field and therefore $\K'$ separable over $\K$, which means (note that $\K \subset \K'$ doesn't need to be algebraic), that
every finitely generated subextension is separably generated  over $\K$ compare \cite[A1.2]{eisenbud:comm_alg}. Whereas (i) and (ii) would work for any field extension (iii) and (iv) are in general wrong if $\K \subset \K'$ 
is not separable.
\end{remark}

\begin{proof}
	(i) and (ii) follow because $\K[\bx] \subset \K'[\bx]$ is a faithfully flat ring extension. (iii) is a consequence of the fact, that any reduced $\K$-algebra is geometrically reduced~\cite[Lemma~10.42.6, Lemma~10.44.6]{stacks}. We will show (iv) with the general jacobian criterion \cite[Thm. 5.7.1]{greuel:singular_commutative}, since there appears to be no reference in the usual literature on commutative algebra.
	
	First choose any associated prime $\frP$ of $\frp_{\K'}$ and let $\hat{\frp}$, $\hat{\frP}$ denote the preimages of $\frp$ and $\frP$ in $\K[\bx]$ and $\K'[\bx]$ respectively. Now write $K$ for the quotient field of $\K[\bx]/\hat{\frp}$ and 
	$K'$ for the quotient field of $\K'[\bx]/\hat{\frP}$. Since $\hat{\frP} \cap \K[\bx] = \hat{\frp}$ \cite[VII Theorem 36]{zariski:comm_alg}, $K$ is clearly subfield of $K'$. 
	For any $K$-vectorspace $V$ we then have $\dim_K V = \dim_{K'} K' \otimes_K V$, since the tensor product commutes with direct sums. Consequently
	\[
		\rank \left[  \frac{\partial f_i}{\partial x_j}  \bmod \hat{\frp} \right]_{i,j} = \rank \left[  \frac{\partial f_i}{\partial x_j}  \bmod \hat{\frP} \right]_{i,j} =: h,
	\]  
	where $\langle f_1, \ldots, f_n \rangle = I$, as stated in the beginning of section~\ref{sec:alg_prelim}.
	
	Now assume $A_{\frp}$ is a regular local ring and choose an associated prime $\frq$ of $I$ with $\frq \subset \hat{\frp}$ (note that there should be only one prime with this property, otherwise $A_{\frp}$ wouldn't be regular). Then we conclude $\height \frq = h$ from the jacobian criterion. Now any associated prime of $\frq_{\K'}$ has height $h$ as well \cite[VII Theorem 36]{zariski:comm_alg} and one of them is contained in $\hat{\frP}$. But then $(A_{\K'})_{\frP}$ is regular according to the general jacobian criterion. 
	
	On the contrary assume $(A_{\K'})_{\frP}$ is regular. Then there exists an associated prime $\frQ$ of $I_{\K'}$ with $\frQ \subset \hat{\frP}$ and $\height \frQ = h$. 
	Now since $\frQ$ is associated to $I_{\K'}$ it is associated to $\frr_{\K'}$ for a primary ideal $\frr \in \K[\bx]$ which is part of a primary decomposition of $I$ (use $(J_1 \cap J_2)\, \K'[\bx] = J_1 \K'[\bx] \cap J_2 \K'[\bx]$ for ideals $J_1,J_2 \leq \K[\bx]$ and $\frQ = (I_{\K'} : \langle b \rangle)$, for some $b \in \K'[\bx]$). So $\frq := \sqrt{\frr}$ is a prime ideal associated to $I$. 
	Now
	\[
	\frr = \frr_{\K'} \cap \K[\bx] \subset \frQ \cap \K[\bx] \subset \hat{\frP} \cap \K[\bx] = \hat{\frp}.
	\]
	But then $\frq = \sqrt{\frr} \subset \hat{\frp}$. Also $h = \height \frQ  \geq \height \frq$. Consequently $A_{\frp}$ is regular according to the general jacobian criterion.	
\end{proof}

%
%

\subsection{Real Algebra}
We review some facts from real algebra. Most of them can be found in \cite{lam:real_alg} or \cite{bochnak:real_alg_geom}.
\begin{definition}
	Let $B$ be any commutative ring and $I \leq B$ an ideal. $B$ is called (formally) \defm{real}, iff any equation
	\[
	b_1^2 + \ldots + b_k^2 = 0, \quad k \geq 1,
	\] 
	implies $b_1 = \ldots = b_k = 0$. $I$ is called \defm{real}, if $B/I$ is real. Also we define the \defm{real radical}
	\[
	\sqrt[r]{I} = \{\,  x \in B \mid x^{2r} + b_1^2 + \ldots + b_k^2 \in I, \text{ for $r,k \geq 0$, $b_i \in B$} \,\},
	\]
	which is the smallest real ideal containing $I$ or $B$ if there are no real ideals between $I$ and $B$ cf.\cite{bochnak:real_alg_geom}. Therefore $I$ is real if and only if $\sqrt[r]{I} = I$.
\end{definition}
The analogue to Hilbert's Nullstellensatz in real algebraic geometry is the
\begin{prop}[Risler's Real Nullstellensatz \cite{lam:real_alg}]\label{prop:real_null}
	Let $I \leq \K[\bx]$ be any ideal. Then 
	\[
	\algi_{\K}(\algrv(I)) = \sqrt[r]{I}.
	\]
\end{prop}
\begin{exs}\label{ex:real}\hfill
\begin{itemize}
\item[(i)] $\C$ is clearly not real, since $1^2 + i^2 = 0$, but $\Q$ and $\R$ are. Also any domain $B$ is real iff its field of fractions is real $Q(B)$ can then be ordered.

\item[(ii)] Consider the ideal $I = \langle x^2 + y^2\rangle \leq \R[x,y,z]$ from Ex.~\ref{ex:introduction}~(iv). Then $I$ is not real, since $x,y \notin I$.
We see easily from the definition, that $x,y \in \sqrt[r]{I}$ and from the real Nullstellensatz~(see Proposition~\ref{prop:real_null}) follows, that $1 \notin \sqrt[r]I$. Hence $\sqrt[r]{I} = \langle x,y \rangle$.

\item[(iii)]
Let $I = \langle x^3 - 5\,y^3 \rangle \leq \Q[x,y]$ from Ex.~\ref{ex:introduction}~(v). Then $I$ is prime in $\Q[x,y]$. Since there exist points $p \in \algrv(I)$ with $(\R[x,y]/I_\R)_{\frm_p}$ regular, $I$ must be real in $\Q[x,y]$, see remark (i) after Proposition~\ref{prop:simple_point}. $I_\R$ is not real however, since $\sqrt[r]{I_\R} = \algi_{\R}(\algrv(I_\R)) = \langle x - \sqrt[3]{5}\,y \rangle$. This is different for the standard radical, see Proposition~\ref{prop:base_change}

\item[(iv)] $f(x,y) = y^3 + 2\,x^2y - x^4$ from Ex.~\ref{ex:introduction}~(vi) is an irreducible polynomial in $\R[\bx]$ and for any $x_0 \ne 0$, there exists a real solution $y_0 \in \R$ of $f(x_0,y) = 0$, since this is a polynomial of degree $3$. Also the local ring at $(x_0,y_0)$ is regular with the jacobian criterion, hence $I = \langle f \rangle$ is a real ideal of $\R[\bx]$ according to~Proposition~\ref{prop:simple_point}.
\end{itemize}	
\end{exs}

\begin{prop}[Simple Point Criterion \cite{lam:real_alg}]\label{prop:simple_point}
Let $\K=\R$ and $I \leq \R[\bx]$. Then $I$ is real, if and only if $I$ is radical and for every associated prime $\frp$ of $I$ there exists $x \in \algrv(\frp)$, with $A_{\frm_x}$ regular.
\end{prop}

\begin{remarks}\hfill
	\begin{itemize}
		\item[(i)] We can easily modify the proof in \cite{lam:real_alg} to show the following (one-sided) generalization for $I \leq \K[\bx]$:
		Assume $I$ is radical and for every associated prime $\frp$ of $I$, there exists $x \in \algrv(\frp)$ with 
		$(A_{\R})_{\frm_x}$ regular, then $I$ is real.
		
		\item[(iii)] There exist algorithm to compute the real radical of an ideal $J \leq \Q[\bx]$ (e.g. \texttt{realrad} in \texttt{Singular}), but to the authors knowledge, all implemented algorithm so far only compute over $\Q$ since there is ambiguity in the ordering of field extension of $\Q$ (in \texttt{Singular} we have \texttt{realrad(x\^{}3 - 5y\^{}3) = x\^{}3 - 5y\^{}3}). 
	\end{itemize}
\end{remarks}

\section{Analytic Preliminaries}
In the following we let $\K = \R,\C$. Any open set $U \subset \K^n$ is meant to be euclidean open. $f$ is called analytic at $p \in K$ (or holomorphic for $\K = \C$), if
\[
f(z) = \sum c_{i_1}, \ldots c_{i_n} \, (z_{1} - p_1)^{i_1} \ldots (z_{n} - p_n)^{i_n}
\]
in a neighborhood of $p$. 

A $d$-dimensional smooth (analytic, complex) \defm{submanifold} of $\K^n$ is a set $X \subset \K^n$ such that for every $p$ in $X$ there exists an open set $U \subset \K^n$ and a $C^\infty$-diffeomorphism ($C^{\omega}$, biholomorphism) $\phi \colon U \to V$ to an open set $V \subset \K^n$, with
\[
X \cap U = \{ x \in U \mid \phi_{d+1}(x) = \ldots = \phi_{n}(x) = 0 \, \} .
\]

A set $X \subset \K^n$ with point $p \in X$ is locally at $p$ \defm{the graph} of an analytic (smooth, holomorphic) mapping (in the first $d$ coordinates), if there exists an open neighborhood $U$ of $p$ and an analytic (smooth, holomorphic) mapping $\psi \colon \rho(U) \to \K^{n-d}$ such that 
\[
X \cap U = \{\, (y, \psi(y)) \mid y \in \rho(U) \,\},
\]
where $\rho \colon \K^n \to \K^d$ is the projection to the first $d$ coordinates. Note that it needs to be checked, that this is a local definition which we leave to the reader.
\begin{prop}\label{prop:manifold_point}
	Let $X \subset \K^n$ be any set and $p \in X$. Then the following conditions are equivalent:
	\begin{itemize}
		\item[(a)] There is an open neighborhood $U$ of $p$ such that $X \cap U$ is an analytic (smooth, holomorphic) submani\-fold of $\K^n$.

		\item[(b)] There exists a permutation $\pi \colon \K^n \to \K^n$ of coordinates such that 
		\[
		\pi(X) = \{\, \pi(x) \mid x \in X \,\} 
		\]	
		is locally the graph of an analytic (smooth, holomorphic) mapping at $\pi(p)$.
		
		\item [(c)] For a generic choice of $A \in \GL(n,\K)$, $A(X)$ is locally the graph of an analytic (smooth, holomorphic) mapping at $Ap$.
	\end{itemize}
\end{prop}

\begin{definition}
A point $x$ of a set $X \subset \K^n$ is an analytic (smooth, holomorphic) \defm{manifold point} of $X$, if any of the equivalent conditions of Proposition~\ref{prop:manifold_point} is fulfilled.
\end{definition}
Any smooth mapping parameterizing a real algebraic set will be a smooth semi-algebraic mapping whose component functions are known to be Nash-functions~\cite[2.9.3]{bochnak:real_alg_geom} and in particular analytic. We get the following proposition:
\begin{prop}\label{prop:alg_ana_smooth}
	Let $\K = \R$ and $X \subset \R^n$ be a real algebraic set, with $p \in X$. $p$ is an analytic manifold point of $X$ if and only if $p$ is	a smooth manifold point of $X$
\end{prop}

In light of Proposition~\ref{prop:alg_ana_smooth} it is enough to work with analytic manifold points if one considers algebraic subsets of $\R^n$. From now on manifold point means analytic/holomorphic manifold point.

%% file: local_real.tex
\section{Local real algebraic geometry}\label{sec:local_real}
We now assume $\K = \R$ and that the singular point of $X$ is at the origin. So we have $I \leq \R[\bx]$ an ideal with $I \subset \langle \bx \rangle =: \frm$ and $A_{\frm}$ not regular, where $A = \R[\bx]/I$. As we have seen in Example~\ref{ex:introduction}~(vi) we need to investigate the extension of $I$ in the ring of convergent power series or the completion of the local ring $A_{\frm}$. The following notations will be used:
\begin{definition}\label{def:list}\hfill
	\begin{itemize}
		\item[(i)] $I_{\frm} = I \R[\bx]_{\frm}, \quad \CR = \R[\bx]_\frm/I_{\frm} = A_{\frm}, \quad \frr = \frm \CR$  
		
		\item[(ii)] $I' = I \R\{ \bx \}, \quad \CO = \R\{ \bx \}/I', \quad \fro = \frm \CO$
		
		\item[(iii)] $I'' = I \R[[ \bx ]], \quad \CF = \R[[ \bx ]]/I'', \quad \frf = \frm \CF$

	\end{itemize}
\end{definition}\noindent
Since the ring extensions $\R[\bx]_\frm \to \R\{\bx\} \to \R[[\bx]]$ are faithfully flat, we have the following chain of local rings
\[
\CR \subset \CO \subset \CF.
\]
We will also need the fact, that $\CF$ is the $\frr$-adic completion of $\CR$:
\[
\CF = \varprojlim_{k} \CR/\frr^k.
\]
Now we define the following ideal of $\R\{\bx\}$, which is usually called the vanishing ideal of the set germ $(X_\R,0)$ \cite{risler:real_nullstellensatz}. We can do a similar construction for $\R[[\bx]]$, but since $f(p)$ is here in general not defined for elements $p \in \R^n$, we need to replace points in $\R^n$ with tuples of formal Puiseux series without constant term,  $f \in \R[[\bx]]$. See \cite[Def. IV.4]{ruiz:power_series} for this approach. 
\begin{definition}\label{def:hat_i}
\[
\hat{I} = \left\{ f \in \R\{\bx\}  \Biggm|  \parbox{18em}{$ \exists \ U \ni 0$ euclidean neighborhood with $f$ \\ converging on $U$ and $f \equiv 0$ on $X_\R \cap U$ } \right\}, 
\quad \hat{\CO}  = \R\{ \bx \}/\hat{I}.
\] 
\end{definition}

\begin{thm}\label{thm:mani_reg}
The origin is a manifold point of $X$ if and only if $\hat{\CO}$ is regular.
\end{thm}

\begin{proof}
	First let the origin be a manifold point of $X$ (of dimension $d$) with parametrization
	\begin{align*}
	\psi \colon U & \mapsto \R^n, \\ 
	(x_1 \ldots x_d) & \to (x_1, \ldots x_d, \phi_1(x_1, \ldots x_d), \ldots, \phi_{n-d}(x_1, \ldots, x_d)),
	\end{align*}
	where $U$ is an euclidean neighborhood of the origin and $\psi(0) = 0$.
	We set 
	\[
	K := \langle x_{d+1} - \phi_1(x_1, \ldots, x_{d}), \ldots, x_{n} - \phi_{n-d}(x_1, \ldots x_{d}) \rangle \leq \R\{\bx\}
	\]
	and claim, that $K = \hat{I}$.
	
	Clearly we have $K \subset \hat{I}$, so let $a \in \hat{I}$. Since $\psi(0) = 0$, we can compose $a$ and $\psi$ and get a converging power series 
	\begin{equation}\label{eq:proof_germ_manifold_smooth}
	a(x_1, \ldots x_d, \phi(x_1, \ldots, x_{d}), \ldots \phi_{n-d}(x_1, \ldots x_{d})) = 0, 
	\end{equation}
	which follows because $a \circ \psi$ is identically zero close to the origin.
	We now set $\psi_i := x_{n + i} - \phi_i(x_1, \ldots x_d) \in \R\{\bx\}$ and have, that $\psi_i$ is of $x_{d+i}$-order $1$. According to the Weierstrass Division Theorem \cite[3.2]{ruiz:power_series} we have a representation
	\[
	a = q_1 \cdot \psi_1 + r,
	\]
	with $q_1 \in \R{\bx}$ and $r \in \crps{x_1, \ldots ,x_d, x_{d+2}, \ldots, x_{n-1}}$. If we iterate this process with $r$ instead of $a$, we have a decomposition
	\[
	a = q_1 \cdot \psi_1 + \ldots + q_{n-d} \cdot \psi_{n-d} + r,
	\]
	with $r \in \crps{x_1, \ldots, x_d}$. Because of \eqref{eq:proof_germ_manifold_smooth} and
	\[
	\psi_i(x_1, \ldots x_d, \phi_1(x_1, \ldots, x_d), \ldots, \phi_{n-d}(x_1, \ldots, x_d)) = \phi_i(x_1, \ldots, x_d) - \phi_i(x_1, \ldots x_d) = 0,
	\]
	we have
	\[
	r(x_1, \ldots, x_d) = 0,
	\]
	so $r = 0$ and therefore $a \in K$. 
	
	It now remains to check, that $\crps{x_1, \ldots x_n}/K$ is a regular local ring. We will use Nagata's Jacobian Criterion~\cite[4.3]{ruiz:power_series}. With $\frm'$ the maximal ideal of $\R\{\bx\}$, it is enough to show, that $\frm' \not \supset J_{n-d}(K)$ and $\height(K) \leq n-d$, where $J_{n-d}(K)$ is the jacobian ideal of order $n-d$ of $K$ \cite[4.1]{ruiz:power_series}. Then $\R\{\bx\}$ is a regular local ring of dimension $n - (n - d) = d$. Since $K$ is generated by $n-d$ elements $\height(K) \leq n-d$ follows easily from Krull's height theorem~\cite[11.16]{atiyah:intro_comm}. Now the jacobian 
	\[
	\frac{D(\psi_1, \ldots, \psi_{n-d})}{D(x_{d+1}, \ldots, x_{n})} = 1,
	\]
	hence $J_{n-d} = \R\{\bx\} \not\supset \frm'$.

	Now on the contrary let $\hat{\CO}$ be regular with $\dim \hat{\CO} = d$. According to Nagata's Jacobian Criterion, it is $\frm' \not\supset J_{n-d}(\hat{I})$. Since clearly $\frm' \supset \hat{I}$ there must exist $g_1, \ldots, g_{n-d} \in \hat{I}$ such, that w.l.o.g
	\[
	\frac{D(g_1, \ldots, g_{n-d})}{D(x_1, \ldots,x_{n-d})} \notin \frm'.
	\]
	But this means, that determinant of the first $(n-d)$ columns of the jacobian matrix of the $g_i$ evaluated at the origin is nonzero. Let $U$ be an euclidean environment of the origin such, that the region of convergence of $g_i$ is contained in $U$ for all $i$. We then set
	\[
	X' := \{\, x \in U \mid g_i(x) = 0,\ \text{for all $i$} \,\}.
	\]
	According to the analytic implicit function theorem the origin is then a manifold point of $X'$ and we only have to show, that 
	$X'$ agrees with $X$ on a neighborhood of the origin, which follows easily if we can prove $K \coloneqq \langle g_1, \ldots, g_{n-d} \rangle = \hat{I}$. 
	
	By our choice of $g_1, \ldots, g_{n-d}$ we clearly have $K \subset \hat{I}$. On the other hand, since $\frm' \not\supset J_{n-d}(K)$ and $\height(K) \leq {n-d}$ we can apply Nagata's Jacobian Criterion again to see, that $\R\{\bx\}/K$ is a regular local ring of dimension $d$. Then $\R\{\bx\}$ is also integral, so $K$ is a prime ideal and because $\R\{\bx\}$ is Cohen-Macaulay we have $\height(K) = \dim \R\{\bx\} - \dim \R\{\bx\}/K = n-d$. But $\hat{I}$ is prime as well with $\height(\hat{I}) = n-d$. Since $K \subset \hat{I}$, we have $K = \hat{I}$. This completes the proof.
\end{proof}

The following theorem due to Risler (Rückert for the complex case) allows the calculation of $\hat{I}$.
\begin{thm}[{{Risler's Real Analytic Nullstellensatz \cite[Th\'{e}or\`{e}me~4.1]{risler:real_nullstellensatz}}}]\label{thm:nullstellensatz}
\[
\hat{I} = \sqrt[r]{I'} = \left\{\,  f \in \R\{ \bx \} \bigm| f^{2n} + b_1^2 + \ldots + b_k^2 \in I', \quad r,k \geq 0,\ b_i \in \R\{\bx\} \,\right\},
\]
\end{thm}


The next proposition collects some well known facts on the relationship of the rings $\CR$, $\CO$, $\CF$.
\begin{prop}\label{prop:compare}\hfill
	\begin{itemize}
		\item[(a)]
		$\CR$ reduced $\Leftrightarrow$ $\CO$ reduced $\Leftrightarrow$ $\CF$ reduced.
		
		\item[(b)]
		$\CR$ normal domain $\Leftrightarrow$ $\CO$ normal domain $\Leftrightarrow$ $\CF$ normal domain.
		
		\item[(c)]
		$\CR$ regular $\Leftrightarrow$ $\CO$ regular $\Leftrightarrow$ $\CF$ regular.
		
		\item[(d)] 
		$I'$ real $\Leftrightarrow$ $I''$ real, i.e. $\CO$ real $\Leftrightarrow$ $\CF$ real.
		
		\item[(e)] $\CO/\sqrt[r]{(0)}$ is regular if and only if $\CF/\sqrt[r]{(0)}$ is regular.
		
		\item[(f)] If $\CF,\CO$ regular, then $\CF,\CO$ real. 
\end{itemize}
\end{prop}

\begin{proof}[proof of Proposition~\ref{prop:compare}]
The proofs for (a),(b),(c),(d) can be found in \cite[ch.~V,VI]{ruiz:power_series}. (e) is also an easy consequence of results in \cite{ruiz:power_series}. We will carry out a proof for completeness sake: 

It is only needed to prove, that $\sqrt[r]{I'}\, \R[[\bx]]  = \sqrt[r]{I''}$. Then the statement follows from one of Nagata's Comparison results~\cite[Prop.~V.4.5]{ruiz:power_series}, $\R[[\bx]] \sqrt[r]{I'} \subset \sqrt[r]{I''}$ is clear, so we proceed to demonstrate $\R[[\bx]] \sqrt[r]{I'} \supset \sqrt[r]{I''}$ by slightly adjusting the proof of Theorem V.4.2 in \cite{ruiz:power_series}. Let $f \in \sqrt[r]{I''}$, which means 
\[
f^{2s} + p_1^2 + \ldots + p_k^2 \in I'',
\]
for elements $p_1, \ldots, p_k \in \R[[\bx]]$. Consequently we have
\[
(\bar{f})^{2s} + \bar{p}_1^2 + \ldots + \bar{p}_k^2 = 0,
\]
in $\CF$. According to M. Artin's Approximation Theorem in the form of \cite[Prop. V.4.1]{ruiz:power_series} we find elements $\hat{f}, \hat{p}_1, \ldots, \hat{p}_k \in \CO$, for every $\alpha \geq 1$ such, that
\[
\hat{f}^{2s} + \hat{p}_1^2 + \ldots \hat{p}_k^2  = 0,
\]
and $\bar{f} = \hat{f} \bmod \frf^{\alpha}$ (recall that $\frf$ is the maximal ideal of $\CF$).  Then for every $\alpha \geq 0$:
\[
f \in \sqrt[r]{I'} \, \R[[\bx]] + (\frm'')^{\alpha},
\]
where $\frm'' = \frm \, \R[[\bx]]$ is the maximal ideal of $\R[[\bx]]$. It follows
\[
f \in \bigcap_{\alpha} (\sqrt[r]{I'} \, \R[[\bx]] + (\frm'')^{\alpha} \R[[\bx]]) = \sqrt[r]{I'}\R[[\bx]],
\]
since any ideal of $\R[[\bx]]$ is closed for the $\frm''$-adic topology.

Now we go on to show (f). Clearly $\CF$ is a regular local ring, with $\CF/\frf \cong \CR/\frr \cong \R$ real. Then $\CF$ must be real
according to \cite[Prop. 2.7]{lam:real_alg}.
\end{proof}

	With some minor modifications all the theory so far in section~\ref{sec:local_real} (except the statements about realness) would also work if we exchange $\R$ with $\C$ and Theorem~\ref{thm:nullstellensatz} with Rückert's analytic Nullstellensatz~\cite[Theorem~2.20, Theorem~3.7]{gunning:analytic_functions}, which states, that $\hat{I} = \sqrt{I'}$ in the complex setting. From Proposition~\ref{prop:compare} we then see easily, why there is usually no need in complex algebraic geometry to consider the completion of $\CR$ to answer questions about the regularity of $\hat{\CO}$. Because then $\hat{\CO} = \CO/\sqrt{(0)} = \CO$ if $\CR$ is reduced, and $\CO$ is regular iff $\CR$ is regular.
	
	In the real case, it is not enough for $I$ to be real to imply the realness of $I'$, see Example~\ref{ex:introduction}~(vi), hence $\hat{I}$ is in general bigger than $I'$ and the nonregularity of $\CR$ does not imply the nonregularity of $\hat{\CO}$. On the other hand if $\CR$ is regular, then $\CO$ is regular and real, hence also $\hat{\CO} = \CO$.

\begin{cor}\label{cor:real_mani}
Let $I''$ or $I'$ be real. The origin is a manifold point of $X = \algrv(I)$, if and only if the origin is nonsingular.
\end{cor}

%% file: normalization.tex
\section{Normalization and Analytic Branches}
In order to decompose the extended ideal $I''$ we look to the normalization of $\CF$, which can be compared to the the normalization of $\CR$.

In this section we assume again $\K=\R$, but now we also require $I \leq \R[\bx]$ to be a radical ideal, with minimal decomposition
\[
I = \frp_1' \cap \ldots \cap \frp_k'.
\]
Now let w.l.o.g $\frp_1', \ldots, \frp_s' \subset \frm = \langle \bx \rangle$ and $\frp_{s+1}', \ldots, \frp'_{n} \not\subset \frm$. In $\CR = (\R[\bx]/I)_{\frm}$
we have consequently then the following minimal decomposition of the zero ideal:
\begin{equation}\label{eq:dec_zero_r}
(0) = \frp_1 \cap \ldots \cap \frp_s,
\end{equation}
where $\frp_i$, $i=1, \ldots, s$ is the prime ideal generated by $\frp_i'$ in $\CR$. From now on we will use the notation $\CR_i = \CR/\frp_i$ and for any
reduced ring $A$ we will write $\ov{A}$ for the integral closure of $A$ in its total ring of fractions. The following lemma collects some
well known facts about the integral closure of reduced local rings.
\begin{lemma}\label{prop:local_normal}
It is 
\[
\ov{\CR} = \ov{\CR_1} \times \ldots \times \ov{\CR_s}
\]
a product of semi-local normal domains. Additionally we have 
\[
\sqrt{\frr \ov{\CR}} = (\frn_{11} \cap \ldots \cap \frn_{1k_1}) \cap \ldots \cap (\frn_{s1} \cap \ldots \cap \frn_{sk_s})
\]
where the $\frn_{ij}$ are the maximal ideals of $\ov{\CR}$ in the form
\[
\frn_{ij} = \ov{\CR_1} \times \ldots \times \ov{\CR_{i-1}} \times \frn_{ij}' \times \ov{\CR_{i+1}} \times \ldots \times \ov{\CR_{s}},
\]
and $\frn_{ij}'$ is one of the $k_i$ maximal ideals of $\overline{\CR_i}$. Also we have the following minimal decomposition 
$\sqrt{\frr \overline {\CR_i}} = \frn_{i1}' \cap \ldots \cap \frn_{ik_i}'$ and
\[
\ov{\CR}_{\frn_{ij}} \cong (\ov{\CR_i})_{\frn'_{ij}}.
\]
\end{lemma}
We now want to compare the normalization of $\CF$ and the completion of $\ov{\CR}$, so we need to investigate what form $\ov{\CR}_{\frn}$ can take for $\frn \leq \ov{\CR}$ maximal. Since $\ov{\CR}_{\frn} = (\ov{\CR_i})_{\frn'}$ for some $i$ and $\frn' \leq \CR_i$ maximal we assume, that $\CR$ is a domain. The 
following exposition is taken from~\cite[VI.4]{ruiz:power_series} and can be checked for details. Since 
\[
\R = \CR/\frr \subset \ov{\CR}/\frn
\]
is an algebraic field extension it must be $\ov{\CR}/\frn = \C,\R$. We distinguish between the following three cases:
\begin{itemize}
	\item[(a)] $\ov{\CR}/\frn = \R$. Since $\ov{\CR}$ is finitely generated over $\CR$, we can extend a surjection $\R[\bx]_{\frm} \to \CR$ to a surjection $\R[\bx,\by]_{\langle \bx,\by \rangle} \to \ov{\CR}_{\frn}$. Hence $\ov{\CR}_{\frn} \cong \R[\bx,\by]_{\langle \bx, \by \rangle}/J$ and its
	formal completion $(\ov{\CR}_{\frn})^* = \R[[\bx,\by]]/J\R[[\bx,\by]]$ is the $\frn$-adic completion of $\ov{\CR}_{\frn}$.  

	\item[(b)] $\ov{\CR}/\frn = \C$ and $\sqrt{-1} \in Q(\CR)$. Since $\ov{\CR}$ is integrally closed, $\C \subset \ov{\CR}$. Then we get a surjection
	$\C[\bx,\by]_{\langle \bx, \by \rangle} \to \ov{\CR}_{\frn}$ and the formal completion 
	$(\ov{\CR}_{\frn})^* = \C[[\bx,\by]]/J\C[[\bx,\by]]$ is the $\frn$-adic completion of $\ov{\CR}_{\frn}$. 
	
	\item[(c)] $\ov{\CR}/\frn = \C$ and $\sqrt{-1} \notin Q(\CR)$. Now we need to adjoin $\sqrt{-1}$ to $\ov{\CR}$ and we get a unique maximal ideal $\frn'$ in $\ov{\CR}[\sqrt{-1}]$ over $\frn$ and the formal completion $(\ov{\CR}_{\frn})^*$ is considered as the formal completion of $(\ov{\CR}[\sqrt{-1}])_{\frn'}$ as in (b). One needs to take care though since this is not the $\frn$-adic completion of $\ov{\CR}_{\frn}$.  	
\end{itemize}

Now we set $\CF_i := \CF/(\frp_i \CF) = \R[[\bx]]/\frp'\R[[\bx ,]]$, for $i=1,\ldots,s$, which is the formal completion of $\CR_i$.  
\begin{prop}[{{Ruiz, Zariski \cite[Prop.~VI.4.4]{ruiz:power_series}}}]\label{thm:normal}
For any $i=1, \ldots, s$ 
\[
\ov{\CF_i} = [(\ov{\CR_i})_{\frn_{i1}}]^* \times \ldots \times [(\ov{\CR_i})_{\frn_{ik_i}}]^*
\]
and $[(\ov{\CR_i})_{\frn_{ij}}]^* \cong \ov{\CF_i/\frq_{ij}}$, where $\frq_{i1}, \ldots, \frq_{ik_i}$ are the associated primes of $(0)$ in 
$\CF_i$. Additionally
\[
\ov{\CF} = \ov{\CF_1} \times \ldots \times \ov{\CF_s}.
\]
\end{prop}

\begin{remark}
The importance of Proposition~\ref{thm:normal} for us lies in the fact, that $\CF$ is real if and only if $\overline{\CF}$ is real, so we can check realness on completions of local rings of normal varieties and use Theorem~\ref{thm:efroymson} for example.
\end{remark}

\begin{proof}
The only thing missing from the proof in \cite{ruiz:power_series} is to take in account non-domains $\CR$, so we need to check
\[
\ov{\CF} = \ov{\CF_1} \times \ldots \times \ov{\CF_s}.
\]
According to Chevalley's Theorem~\cite[Prop.~VI.2.1]{ruiz:power_series} we have a minimal decomposition
\[
\frp_i \CF = \frq_{i1}' \cap \ldots \cap \frq_{ik_i}',
\]
with $\frq_{ij}$ prime of height $\height \frp_i =: d_i$ and $\frq_{ij} = \frq_{ij}' \CF_i$. It only remains to show, that
\[
(0) = (\frp_1 \cap \ldots \cap \frp_s) \CF = \frq_{11}' \cap \ldots \cap \frq_{1k_1}'  \ldots \ \frq_{s1}' \cap \ldots \cap \frq_{sk_s}'
\]
is a minimal decomposition of $(0)$ in $\CF$, because then
\[
\overline{\CF} = \bigtimes_{i,j} \overline{\CF/\frq_{ij}'} =  \ov{\CF_1} \times \ldots \times \ov{\CF_s}.
\]
Now suppose w.l.o.g 
$\frq_{11}' \supset \bigcap_{i,j \ne 1,1} \frq_{ij}'$.
Then because $\frq_{11}'$ prime, there exists $\frq_{ij}' \subset \frq_{11}$, where clearly $i \ne 1$. If we can show, that 
$\frq_{ij}' \cap \CR = \frp_i$, we are done, since \eqref{eq:dec_zero_r} is a minimal decomposition. 

Assume $\frp_i \subsetneq \fra \coloneqq \frq'_{ij} \cap \CR$. Then since $\fra$ is prime it is $\height \fra > d_i = \height \frp_i$.
Consequently according to Chevalley's Theorem every associated prime of $\fra \CF$ is of height greater $d_i$. Since $\fra \CF \subset \frq'_{ij}$ and $\height \frq'_{ij} = d_i$, this is a contradiction. 
\end{proof}

\section{Real Algebraic Curves}
Now we will apply the theory of the last section to singularities of real algebraic curves. Let $\dim I = 1$, then the analysis of 
$\hat{\CF} = \CF/\sqrt[r]{(0)}$ will be especially satisfying, since the real radical of an associated prime $\frq$ of $I''$ 
will be either $\frq$ itself or the maximal ideal $\frm'' = \frm \R[[\bx]]$ of $\R[[\bx]]$:
\begin{lemma}\label{lm:ideal_height_real}
	Let $\frq \leq \R[[\bx]]$ be a prime with $\height \frq = n-1$. Then 
	\[
	\sqrt[r]{\frq} = \begin{cases}
	\frq & \text{$\frq$ real}\\
	\frm'' & \text{$\frq$ not real}
	\end{cases}
	\]
\end{lemma}

\begin{lemma}\label{lm:local_ring_normal}
	Let $(A,\frm_A) \subset (B,\frm_B)$ be a finite extension of local rings. Assume $\frm_A B = \frm_B$ and
	\[
	B/\frm_B = A/\frm_A
	\] 	
	Then $A = B$.
\end{lemma}

\begin{proof}
	For any element $b \in B$, there exists $a \in A$ with
	$b - a \in \frm_B = \frm_A B$ since $A/\frm_A = B/\frm_B$. It follows
	\[
	B = A + \frm_A B. 
	\]
    Since $m_A$ is the Jacobson-radical of $A$ and $B$ is a finite $A$-module, the statement of the lemma follows from the Lemma of Nakayama.
\end{proof}

We can now formulate the main result of this section
\begin{thm}\label{thm:curves}
	Let $\dim I = 1$ and $I$ radical. The origin is a manifold point of $\algrv(I)$ if and only if one of the following two conditions is true
	\begin{itemize}
		\item[(a)] There is exactly one real maximal ideal $\frn \leq \overline{\CR}$ lying over $\frr = \frm \CR$ and $\frn$ is an isolated primary component of $\frr \overline{\CR}$.
		
		\item[(b)] All the maximal ideals $\frn \leq \overline{\CR}$ are not real. In this case the origin is an isolated point of $\algrv(I)$.
	\end{itemize}
\end{thm}

\begin{proof}
	First let $(0) = \frq_1 \cap \ldots \cap \frq_r$ be a primary decomposition in $\CF$. According to 
	Chevalley's Theorem $\CF$ is reduced~\cite[Prop.~2.1]{ruiz:power_series}, hence all the $\frq_i$ are prime. 
	
	Write now $\frq'_i$ for the ideal of $\CF$ with $\CF/\frq_i = \R[[\bx]]/\frq'$.
	According to Theorem~\ref{thm:normal} there exist $\frn_1, \ldots, \frn_q$ maximal in $\overline{\CR}$ 
	with 
	\[
		\overline{\CF/\frq_i} \cong (\overline{\CR}_{\frn_i})^*
	\]
	First we will show the following statement
	\begin{equation}\label{eq:st_proof_curves}
	\CF/\frq_i \text{ real} \Leftrightarrow \frn_i \text{ real}.
	\end{equation}
	Clearly $\overline{\CF/\frq_i}$ is real if and only if $\CF/\frq_i$ is real, since they are contained in the quotient field of $\CF/\frq_i$, so we need to show, that $(\overline{\CR}_{\frn_i})^*$ is real if and only if $\frn_i$ is real. 
	
	If $\frn_i$ is not real, then $\overline{\CR}/\frn_i \cong \C$ and one can see from the construction before Theorem~\ref{thm:normal} that $(\overline{\CR}_{\frn_i})^*$ will not be real (since $\C \subset (\overline{\CR}_{\frn_i})^*$).
	
	On the other hand let $\frn_i$ be real, then $(\overline{\CR}_{\frn_i})^*$ will be the $\frn_i \overline{\CR}$-adic completion of the local ring $\overline{\CR}_{\frn_i}$. Since $\overline{\CR}$ is normal of dimension $1$, we also know, that
	$\overline{\CR}_{\frn_i}$ is regular according to Serre's regularity criterion $R_1$ \cite[Theorem~39]{matsumura:comm_alg}.
	Then $(\overline{\CR}_{\frn_i})^*$ is regular too, with residue field $\overline{\CR}/\frn_i = \R$. Now $(\overline{\CR}_{\frn_i})^*$ must be real because of \cite[Prop. 2.7]{lam:real_alg}.
	
	We consider now
	\begin{equation}\label{eq:rera_dec}
	\sqrt[r]{I''} = \sqrt[r]{\frq_1'} \cap \ldots \cap \sqrt[r]{\frq_i'}, 
	\end{equation}
	where $\frq_i'$ is the preimage of $\frq_i$ in $\R[[\bx]]$. As one checks easily $\frq_i$ is real if and only if $\frq_i'$ is real.
	
	If none of the $\frn_i$ is real, then none of the $\frq_i'$ is real and according to Lemma~\ref{lm:ideal_height_real}, we would get
	$\sqrt[r]{I''} = \frm''$ from \eqref{eq:rera_dec} and $\hat{\CF} = \R[[\bx]]/\sqrt[r]{I''} \cong \R$ is regular. Since $\hat{I} = \sqrt[r]{I'} = \sqrt[r]{I''} \cap \R\{\bx\}$ (see the proof of Proposition~\ref{prop:compare}~(e)) we would also have $\hat{I} = \frm' = \frm \R\{\bx\}$ and one checks easily with Definition~\ref{def:hat_i}, that the origin must be an isolated point of $X = \algrv(I)$.
	
	If two of the $\frn_i$ are real, then $\hat{\CF} = \R[[\bx]]/\sqrt[r]{I''}$ would not be a domain and therefore not regular. 
	Then the origin cannot be a manifold point of $X$ according to Theorem~\ref{thm:mani_reg}. 
	
	Now we investigate the case, that exactly one $\frn_i$ is real, w.l.o.g. we choose $\frn_1$ real. Then $\sqrt[r]{I''} = \frq_1$.
	We have the following commutative diagram:
	
	\begin{equation}\label{diag:curves_proof}
	\begin{tikzcd}
	\overline{\CR}_{\frn_1} \ar{rr}{\psi} & & \left (\overline{\CR}_{\frn_1} \right)^* \ar{d}{\cong}[swap]{\eta}  \\
	\overline{\CR} \ar{u}{l} \ar{r} & \overline{\CF} \ar{r}& \left( \overline{\CF/\frq_1}\right) \ar{u}	\\
	\CR \ar{u} \ar{r} & \CF \ar{r} \ar{u} & \CF/\frq_1 \ar{u}{\iota}
	\end{tikzcd}
	\end{equation}
	First we assume that $\frn_1$ is an isolated primary component of $\frr \overline{\CR}$. We proceed in several steps.
	
	(1) $\gamma(\frr) \overline{\CR}_{\frn_1} = \frn_1 \overline{\CR}_{\frn_1}$, where $\gamma \colon \CR \to \overline{\CR}_{\frn_1}$.
	Since $\frn_1$ is an isolated prime of $\frr \overline{\CR}$ we find a minimal primary decomposition
	\[
	\frr \overline{\CR} = \frn_1 \cap \frs_2 \cap \ldots \cap \frs_k.
	\]
	For any $x \in \frn_1 \overline{\CR}_{\frn_1}$ we have
	$x = a \cdot \frac{p}{q}$ with $p,q \in \overline{\CR}$, $a \in \frn_1$ and $q \notin \frn_1$. Now choose 
	$f_i \in \frs_i \backslash \frn_1$. Then 
	\[
	b \coloneqq a \cdot f_2 \cdots f_k \in \frr \overline{\CR}
	\]
	and
	\[
	x = a \cdot \frac{p}{q} = b \cdot \frac{1}{f_2} \cdots \frac{1}{f_k} \cdot \frac{p}{q} \in \gamma(\frr) \overline{\CR}_{\frn_1}.
	\]
	
	(2) $\iota(\fro)$ generates the maximal ideal of $\overline{\CF/\frp_1}$, where $\fro$ is the maximal ideal of $\CF/\frq_1$. 
	Since $\psi$ is the $\frn_1 \overline{\CR}_{\frn_1}$-adic completion of $\overline{\CR}_{n_1}$ we know that $\psi(\gamma(\frr))$ generates the maximal ideal of $(\ov{\CR}_{\frn_1})^{*}$. But $\psi(\gamma(\frr)) = \eta(\iota(\fro))$  and we conclude that $\iota(\fro)$ generates the maximal ideal of $\overline{\CF/\frp_1}$.
	
	(3) $\CF/\frp_1$ is regular. We have already seen, that the residue field of $(\overline{\CR}_{\frn_1})^*$ is $\R$, hence the same is true of the residue field of $\ov{\CF/\frp_1}$. Also we know, that $\ov{\CF/\frp_1}$ is finite over $\CF/\frp_1$ \cite[Prop.~III.2.3]{ruiz:power_series} and in (2) we have seen, that the maximal ideal of $\CF/\frp_1$ generates the maximal ideal of $\overline{\CF/\frp_1}$. 
	
	Now we are exactly in the situation of Lemma~\ref{lm:local_ring_normal} with $A = \CF/\frp_1$ and $B = \overline{\CF/\frp_1}$. It follows $\ov{\CF/\frp_1} = \CF/\frp_1$. Then $\CF/\frp_1$ is a normal local ring of dimension at most $1$. With Serre's regularity criterion $R_1$, we see, that $\hat{\CF} = \CF/\frp_1$ is regular and according to Theorem~\ref{thm:mani_reg} the origin must be a manifold point of $X = \algrv(I)$.

	Now suppose on the contrary that $\CF/\frp_1$ is regular and $\frn_1$ is not an isolated primary component of $\frr \overline{\CR}$.
	Since $\CF/\frp_1$ is regular, it is a Cohen Macaulay domain. It fulfills $S_2$ and $R_1$ and is normal by Serre's normality criterion~\cite[Theorem~39]{matsumura:comm_alg}. Therefore $\CF/\frp_1 = \overline{\CF/\frp_1} \cong (\overline{\CR}_{\frn_1})^*$.
	
	Let $\frb$ be the ideal generated by $\gamma(\frr)$ in $\overline{\CR}_{\frn_1}$. Because diagram~\eqref{diag:curves_proof} commutes and $\iota$ is an isomorphism we have that $\psi(\frb)$ generates the maximal ideal $\fra$ of $(\overline{\CR}_{\frn_1})^{*}$. But since $\psi$ is faithfully flat we have
	\[
	\frb = \psi(\frb)(\overline{\CR}_{\frn_1})^* \cap \overline{\CR}_{\frn_1} = \fra \cap \overline{\CR}_{\frn_1} = \frn_1 \overline{\CR}_{\frn_1}. 
	\]
	Since $\frn_1$ is not an isolated primary component of $\frr \overline{\CR}$, there exists a primary ideal
	$\frs$ with $\frr \overline{\CR} \subset \frs \subsetneq \frn_1$ (remember that $\sqrt{\frr \overline{\CR}}$ is the intersection of all maximal ideals of $\overline{\CR}$). 
	But $\frn_1 \overline{\CR}_{\frn_1} = \frb = \langle l(\frr \overline{R}) \rangle \subset \frs \overline{\CR}_{\frn_1}$. Therefore
	\begin{equation}\label{eq:proof_curve}
	\frs \overline{\CR}_{\frn_1} = \frn_1 \overline{\CR}_{\frn_1}.
	\end{equation}
	Now choose $r \in \frn_1 \backslash \frs$. Because of \eqref{eq:proof_curve} there exist $p,q,s \in \overline{\CR}$ with $q \notin \frn_1$, $s \in \frs$ and 
	\[
	\frac{s \, p}{q} = r
	\]
	Thus there is $q' \notin \frn_1$ with $q'\,r \, q = q'\,c \, p \in \frs$, which is primary. Because $r \notin \frs$ it must be $(q\,q')^k \in \frs \subset \frn_1$. But then $q \, q' \in \frn_1$, a contradiction. This completes the proof.
\end{proof}

\begin{ex}
We can test the conditions of Theorem~\ref{thm:curves} with all CAS which have a normalizing algorithm for polynomial rings implemented.
Consider the following run in \texttt{Singular} to test whether the origin is a manifold point of $X = \algrv(y^3 + 2yx^2 - x^4)$:
\begin{lstlisting}[columns=fullflexible, basicstyle=\footnotesize, language=bash]
> ideal I = y^3 + 2*y*x^2 - x^4;
> def nor = normal(I);
> def S = nor[1][1];
> setring S;
> ideal M = norid + ideal(x,y);
> primdecGTZ(M);
[1]:
   [1]:
      _[1]=T(2)
      _[2]=y
      _[3]=x
      _[4]=-T(2)^2+T(1)-2
   [2]: -- same
[2]:
   [1]:
      _[1]=T(2)^2+2
      _[2]=y
      _[3]=x
      _[4]=-T(2)^2+T(1)-2
   [2]: -- same
\end{lstlisting} 
\end{ex}
With $A = \R[x,y]/\langle y^3 + 2yx^2 - x^4 \rangle$, we see, that $\frm \overline{A} = \frn_1' \cap \frn_2'$, where $\frn_1'$ is real and $\frn_2'$ not.
It follows easily, that there is exactly one real ideal $\frn_1$ lying over $\frr \overline{\CR}$ and $\frn_1$ is an isolated prime
of $\frr \overline{\CR}$. From Theorem~\ref{thm:curves} we deduce, that the origin is a manifold point of $X$.

%% file: four_bar.tex
\section{C-Space Singularities of the Four Bar mechanism}\label{sec:four_bar}
Recently efforts have been made in the kinematics community to define and categorize kinematic singularities of linkages in a rigorous way \cite{muller:local_geometry}, \cite{muller:higher_order}, \cite{muller:sing_conf}. It has been observed~\cite[Ex. 6.3.4]{muller:sing_conf} that closed 6R-chains exist with rank drop in the constraint equation but smooth configuration spaces nevertheless. This makes it necessary to decide for singular points in the configuration space whether it is a C-Space Singularity, which are defined as non-manifold points of the configuration space~\cite{muller:sing_conf}. Compare also \cite[p. 227]{piippo:planar}, where this question is asked for some well known planar linkages. 

In this section we would like to apply some of the theory developed so far to the example of Four Bar Linkages.
Conditions on the Design Parameters such, that there exists points with a rank drop in the constraint equations are well known, see e.g. \cite{piippo:planar}, \cite{bottema:theoretical_kin} (Grashof Criterion). Lesser known are methods to show, that these points are C-Space Singularities, i.e. non-manifold points. We will be able to show this for all mechanism in the class of singular four bars with computational methods.


The Four Bar mechanism is one of the oldest and most widely used planar mechanism in Kinematics and Mechanical Engineering. It is also one of the first examples, where singularities in the configuration space were described and analyzed in a systematic way \cite{gosselin:singularities}. In its basic form it consists of four bars connected in a circular arrangement by rotational joints with one bar fixed to the ground:
\begin{center}
	\begin{tikzpicture}[scale=3.5,
	gelenk/.style={circle, draw=blue!50, fill=white, thick, inner sep=0pt,minimum size=1.6mm},
	akt_gelenk/.style={circle, draw=blue!50, fill=white!80, thick, inner sep=0pt,minimum size=2.2mm}]

	\draw[->] (0,0) -- (1.5,0) node[right,fill=white] {};
	\draw[->] (0,0) -- (0,1) node[above,fill=white] {};

	\node[akt_gelenk,label=below:{\scriptsize $A$}] (base) at (0,0) {}; 
	\node[akt_gelenk,label=below:{\scriptsize $B$}] (base2) at (1,0) {};
	\node[gelenk,label={[yshift=2.8pt,xshift=-1pt]below right:{\scriptsize $(x,y)$}}] (g1) at (0.171301045, 0.46974030) {} ;
	\node[gelenk,label={[yshift=2.8pt,xshift=-1pt]below right:{\scriptsize $(u,v)$}}] (g2) at (1.0192522, 0.9998146) {} ;
	
	\node[gelenk] (b2) at (-0.24954, 0.4332779)    {};
	\node[gelenk] (c2) at (0.62099022, 0.92539266) {};
	
	\draw[-,thick] (base) -- node[right,xshift=2pt] {\scriptsize $l_2$} (g1) -- node[below,yshift=-3pt] {\scriptsize $l_4$} (g2) -- node[right, xshift=1pt]{\scriptsize $l_3$} (base2);

	\draw[-,thick, dashed] (base) -- (b2) -- (c2) -- (base2);	
	\end{tikzpicture}
\end{center}

The configuration space, defined as the set of all possible assembly configuration, can be represented
by the real algebraic set $X = \algrv(I)$, where $I = \langle p_1,p_2,p_3 \rangle \leq \R[x,y,u,v]$ is generated by the polynomials
\begin{align*}
p_1 & = x^2 + y^2 - l_2^2,\\
p_2 & = (u - 2)^2 + v^2 - l_3^2,\\
p_3 & = (u - x)^2 + (v - y)^2 - l_4^2.
\end{align*}
$l_2,l_3,l_4$ are the parameters of the four bar which are assumed to be positive real numbers. We fixed the length $l_1 = |AB| = 2$ of the ground bar, since any other length can be treated by scaling the system. 
\paragraph{Dimension of $I$}
We will assume $l_2 \ne 2$, $l_4 \ne 2$, since the complementary case can be analyzed in the same way.
Now we calculate a pseudo Gr\"obner basis of $I$ with respect to the polynomial ordering $(dp(2),dp(2))$ and the enumeration $v,y,u,x$. We can do all the calculations in $B = \Q(l_2,l_3,l_4)[x,y,u,v]$ but we have to be careful to avoid dividing by elements of $\Q(l_2,l_3,l_4) \backslash \Q$ in all Gr\"obner base calculations, since these could be zero for valid parameters $l_2,l_3,l_4$. In \texttt{Singular} we can achieve this by setting
\texttt{option(intStrategy)} and \texttt{option(contentSB)}. 
 
We get $6$ polynomials $g_1, \ldots, g_6$, with the leading terms
\begin{alignat*}{2}
\leadt(g_1) & = -16\,u^2\,x & \leadt(g_4) & =  y^2 \\
\leadt(g_2) & = (-2\,l_2^2+8)\,v\,u \qquad & \leadt(g_5) & = 2\,v\,y \\
\leadt(g_2) & = -2\,v\,x^2 & \leadt(g_6) & = v^2
\end{alignat*}
According to Exercise~2.3.8 of \cite{greuel:singular_commutative} $\{g_1, \ldots, g_6\}$ is a Gr\"obner basis of $I$ as long as $l_2 \ne \pm 2$ which we assumed in the beginning but then we can calculate the dimension of $I$ with
\[
\dim I = \dim \, \langle u^2\,x, v\,u, v\,x^2, y^2, v\,y, \, v^2 \rangle. 
\]
With a simple combinatorial argument \cite[Prop. 9.1.3]{cox:ideals} we see, that the dimension of the right ideal is $1$ and consequently $\dim I = 1$. 
Since $I$ can be generated by the $3$ elements $p_1,p_2,p_3$, $A := \R[x,y,u,v]/I$ must be a complete intersection ring and consequently equidimensional Cohen-Macaulay~\cite[Proposition~18.13]{eisenbud:comm_alg}.

\paragraph{Singular Locus}
According to \cite{piippo:planar} there only exist singular points in $X$, iff 
\[
l_2 \pm l_3 \pm l_4 = 2.
\]
We restrict our investigation to the case $l_2 - l_3 + l_4 = 2$, i.e. $l_3 = l_2 + l_4 - 2 > 0$, since other cases can be handled in a similar way. Since $\dim I = 1$ equidimensional we need to analyze the ideal $J$ generated by $I$ and all the $3$-minors of the jacobian of $(p_1, p_2, p_3)$. With a \texttt{Singular} Gr\"obner base computation we get $J = \langle s_1, s_2,s_3,s_4 \rangle$, with
\begin{align*}
s_1 & = q_1(l_2,l_4)\,x + c_1(l_2,l_4) \\
s_2 & = q_2(l_2,l_4)\,u + r_2(l_2,l_4)\,x + c_2(l_2,l_4) \\
s_3 & = q_3(l_2,l_4)\,y \\
s_4 & = q_4(l_2,l_4)\,v + f(l_2,l_4,x,y,u),
\end{align*}
where all the coefficients are polynomials in $l_2,l_4$ or $l_2,l_4,x,y,u$ respectively.
We need to carefully examine the coefficients of the leading monomials of the $s_i$ to make sure that $\{s_1,s_2,s_3,s_4\}$ is a Gr\"obner basis of $J$ in $A$. 

A quick calculation in \texttt{Singular} shows:
\begin{align*}
q_1(l_2,l_4) & = l_4^2 \cdot (l_4-2) \cdot (l_2+l_4) \cdot (l_2+l_4-2)^2 \cdot (l_2 + 2)^2 \cdot l_2 \cdot (3\,l_2 - 8), \\
q_2(l_2,l_4) & = l_4^2 \cdot (l_2+2 \,l_4-2) \cdot (l_2+l_4-2)^2 \cdot (l_2+2)^2 \cdot (l_2 - 2), \\
q_3(l_2,l_4) & = l_4^2 \cdot (l_2 + l_4 - 2)^2 \cdot (l_2 + 2), \\
q_4(l_2,l_4) & = l_2^2 \cdot (l_2 + 2) \cdot (l_2 - 2),
\end{align*}
Taking in account our assumptions, that $l_4,l_2 > 0$, $l_2 + l_4 - 2 = l_3 > 0$, $l_2 \ne 2$, $l_3 \ne 2$ and in addition $l_2 \ne \frac{8}{3}$ (which we will also need to check separately), we see, that none of the $q_i$ will vanish and $s_1, s_2,s_3,s_4$ forms a Gr\"obner basis of $J$ for all possible values of $l_2, l_4$.
Clearly then $\dim J = 0$ and since $A$ is Cohen-Macaulay, we can infer from \cite[Theorem~18.15]{eisenbud:comm_alg}, that $I$ must be a radical ideal. But then the Singular Locus of $I$ is given by all the prime ideals containing $J$. 

We now set $p = (l_2, 0, l_2 + l_4,0) \in \R^4$. As we can check quickly by substitution in $(s_1,s_2,s_3,s_4)$, we have $J \leq \frm_p$, so $p$ is the only singularity of $X_\C = \algv(I)$. 

\paragraph{Manifold Points} To check whether $p$ is a non manifold point with Theorem~\ref{thm:curves} we need to calculate the integral closure $C$ of
$A_{\frm_p}$. We could do this by applying the normalization algorithm described in \cite{greuel:singular_commutative} and implemented in \texttt{Singular} but it has proven difficult to check the validness of the Gr\"obner base calculations in each step for the considered values of $l_2,l_4$. We could still analyze the situation for generic values of $l_2,l_4$ but we want a statement for all valid values.

Instead we will determine the blow up $\pi \colon \tilde{X} \to X$ at $p$, since $\tilde{X}$ will be nonsingular after one blow up and is then the normalization of $X$.

First we move $p$ to the origin and consider $I_{\textrm{bl}} = \langle p'_1,p'_2,p'_3,b_1,b_2,b_3,b_4 \rangle \leq \R[x,y,u,v,\hat{x}, \hat{y},\hat{u},\hat{v}]$ given by
\begin{align*}
p'_1& = p_1(x + l_2,y,u + l_2 + l_4,v) = x^2+y^2+ 2 l_2 x, \\
p'_2& = p_2(x + l_2,y,u + l_2 + l_4,v) = u^2+v^2+ (2 l_2 + 2 l_4-4) u, \\
p'_3& = p_3(x + l_2,y,u + l_2 + l_4,v) = x^2+y^2-2 x u+u^2-2 y v+v^2 - 2 l_4 x + 2 l_4 u,
\end{align*}
and the homogeneous polynomials
\begin{alignat*}{2}
b_1 & = x\,\hat{y}  - y\, \hat{x}, \qquad \qquad& b_4 &= y\, \hat{u} - u \, \hat{y},\\
b_2 &= x\, \hat{u} - u\, \hat{x},  & b_5 &= y\,\hat{v} - v \, \hat{y},\\
b_3 &= x\, \hat{v} - v \, \hat{x}, & b_6 &= u\, \hat{v} - v \, \hat{u}.
\end{alignat*}
Then we go to the chart $\hat{y} = 1$ and get the system
\begin{align*}
p''_1 & = y \cdot (y \, \hx^2+y+(2 l_2) \hx), \\
p''_2 & = y \cdot (\hu^2+y \hv^2+(2 l_2+2 l_4-4)) \hu,\\
p''_3 & = y \cdot (y \,\hx^2-2 y \hx \hu+y \hu^2+y \hv^2-2 y \hv+y+(-2 l_4) \hx+(2 l_4) \hu).
\end{align*}
We set $I_y := \langle p''_1/y, p''_2/y, p''_2/y \rangle \leq \R[\hx,y,\hu,\hv]$. To get the equations of the strict transform on this chart, we need to remove the exceptional divisor, so we have to calculate the saturation 
\[
J := (I_y : \langle y \rangle^{\infty}).
\]
This can easily be achieved with the command \texttt{sat} in \texttt{Singular}. But again we have to be careful to check whether the Gr\"obner basis calculations stays valid for all assumed values for $l_2,l_4$, so we will calculate the saturation manually.
First we calculate $I_y \cap \langle y \rangle$, which we get by eliminating $t$ of 
\[
I_y \, t + \langle (1-t)y \rangle  
\] 
Now we divide any generator of $I_y \cap \langle y \rangle$ by $y$ and after checking, that all coefficients of the leading monomials won't be zero after substitution of values for $l_2,l_4$ we normalize the generators and get the following Gr\"obner basis of $J = (I_y:\langle y \rangle)$:
\scriptsize
\begin{align*}
f_1 & =\hu^2  \hx^2+\frac{-2  l_2-2  l_4}{l_2+2}  \hu  \hx^3+ \frac{l_2^2+2  l_2  l_4+ l_4^2}{l_2^2+4  l_2+4}  \hx^4+\frac{l_2^2-4  l_2+4}{l_2^2+4  l_2+4}  \hu^2+ \frac{-2  l_2^2-2  l_2  l_4+4  l_2-4  l_4}{l_2^2+4  l_2+4}  \hu  \hx+\frac{l_2^2+2  l_2  l_4+ l_4^2}{l_2^2+4  l_2+4}  \hx^2, \\
f_2 & =y  \hx^2+y+(2 \, l_2)  \hx, \\
f_3 & =y  \hu^2-y  \hu  \hx+ \frac{l_2^2+4  l_2+4}{4}  \hu^2  \hx+ \frac{- l_2^2- l_2  l_4-2  l_2-2  l_4}{2}  \hu  \hx^2+ \frac{l_2^2+2  l_2  l_4+ l_4^2}{4}  \hx^3, \\
f_4 & = \hv  \hx+ \frac{l_2+2}{2  l_2}  \hu  \hx^2+ \frac{- l_2- l_4}{2  l_2}  \hx^3+ \frac{- l_2+2}{2  l_2}  \hu + \frac{- l_2- l_4}{2  l_2}  \hx, \\
f_5 & = \hv  \hu+ \frac{-3  l_2^2-3  l_2  l_4+6  l_2-2  l_4}{ l_2^2-4  l_2+4}  \hv  \hx+ \frac{ l_2^2+4  l_2+4}{2  l_2^2-4  l_2}  \hu^2  \hx^3+ \frac{- l_2^2- l_2  l_4-2  l_2-2  l_4}{ l_2^2-2  l_2}  \hu  \hx^4+  \frac{ l_2^2+2  l_2  l_4+ l_4^2}{2  l_2^2-4  l_2}  \hx^5+ \frac{3  l_2^2+4}{2  l_2^2-4  l_2}  \hu^2  \hx \\ & \quad + \frac{-4  l_2^3-4  l_2^2  l_4+6  l_2^2-2  l_2  l_4+4  l_2+4  l_4}{ l_2^3-4  l_2^2+4  l_2}  \hu  \hx^2+ \frac{5  l_2^3+10  l_2^2  l_4-10  l_2^2+5  l_2  l_4^2-12  l_2  l_4-2  l_4^2}{2  l_2^3-8  l_2^2+8  l_2}  \hx^3+ \frac{2  l_2^2+4  l_2  l_4-4  l_2+2  l_4^2-4  l_4}{ l_2^2-4  l_2+4}  \hx, \\
f_6 & = \hv y+y  \hu  \hx+ \frac{-2  l_2- l_4+2}{2  l_2} y  \hx^2+ \frac{-2  l_2- l_4+2}{2  l_2} y+ ( l_2-2)  \hu+(- l_2+2)  \hx, \\
f_7 &= \hv^2+ \frac{-2  l_2-2  l_4+4}{ l_2-2}  \hv+ \hu^2+ \frac{-2  l_2-2  l_4+4}{ l_2-2}  \hu  \hx+ \frac{ l_2^2+2  l_2  l_4-2  l_2+ l_4^2-2  l_4}{ l_2^2-2  l_2}  \hx^2+ \frac{ l_2^2+2  l_2  l_4-2  l_2+ l_4^2-2  l_4}{l_2^2-2 l_2},
\end{align*}
\normalsize
If we try to repeat the process we get the same ideal, hence $J = (I:\langle y \rangle^{\infty})$ is the ideal of the strict transform of $X$ on the chart $\hy = 1$. Now we check, that the ideal of the $3$-minors of the jacobian of $(f_1, \ldots, f_7)$ is the whole ring and $\tilde{X}$ nonsingular. We can do this as before in the calculation of the singular locus of $X$.

Now we need to identify all points $q$ in the fiber over the origin, so we calculate a pseudo Gr\"obner basis of 
$J + \langle y \rangle$ and get
\begin{align*}
g_1 &= (2 \, l_2) \, \hx, \\
g_2 &= (l_2-2) \,  \hu + (-l_2+2) \, \hx, \\
g_3 &=  y, \\
g_4 &=\hv^2+ \frac{-2  l_2-2  l_4+4}{ l_2-2} \,  \hv+\hu^2+\frac{-2  l_2-2  l_4+4}{ l_2-2} \, \hu \hx \\ &\quad + \frac{ l_2^2+2  l_2  l_4-2  l_2+ l_4^2-2  l_4}{ l_2^2-2  l_2} \hx^2+\frac{ l_2^2+2  l_2  l_4-2  l_2+ l_4^2-2  l_4}{ l_2^2-2  l_2}
\end{align*}
After checking, that this is a Gr\"obner base for all assumed values of $l_2,l_4$ we substitute $\hx=0$, $\hu=0$ from $g_1,g_2$ into $g_4$ and multiply with $(l_2^2 - 2 \,l_2)$. Then we get
\[
g'(\hv) = (l_2^2 - 2\,l_2) \, \hv^2 - l_2\,(2\,l_2 - 2\,l_4 + 4) \hv + (l_2^2+2  l_2  l_4-2  l_2+ l_4^2-2  l_4).
\]
$g'$ is a quadratic equation in $\hv$ with discriminant 
\[
8 \, l_2 \, l_4 \, (l_2 + l_4 - 2) = 8 \, l_2 \, l_3 \, l_4 > 0. 
\]
Consequently all points lying over the origin in the chart $\hy = 1$ are real. One can check analogously that for all other charts the same points (if any) are lying over the origin and therefore the extension $I_{p} \, \R[[x,y,u,v]]$ must be real too, where $I_p$ is the translated ideal. It follows, that $(l_2,0,l_2 + l_4,0)$ is not a manifold point of $X$.

%% file: higher_dim.tex
\section{Higher Dimensions}\label{sec:higher_dim}
If $X$ is of dimension greater than one, it is difficult in general to analyze $\hat{I}$, since we can't effectively compute in the ring of power series. However we can use a criterion by Efroymson to check if $I''$ is real:
\begin{thm}[Efroymson~\cite{efroymson:local_reality}]\label{thm:efroymson}
Let $I$ be a real prime and $\CR$ integrally closed. $I''$ is real, if and only if the origin is contained in the euclidean closure of the real nonsingular points of $X$. 
\end{thm}
We can use this criterion in many cases to decide if a singular real point of $X$ is a non-manifold point, since according to~Corollary~\ref{cor:real_mani} it is enough to show, that the extension of $I''$ in $\R[[\bx]]$ is real at this point. We want to demonstrate this on the configuration space $X$ of the 3RRR-parallel linkage from figure~\ref{fig:3rrr}, but the same arguments can be used for a large class of linkages.
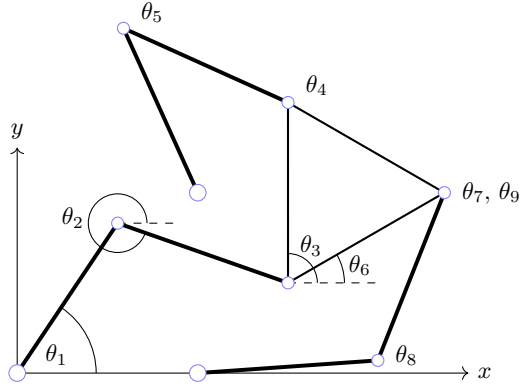
\begin{figure}
		\centering
		\begin{tikzpicture}
		[scale=1.5, joint/.style={circle, inner sep=0pt, draw=blue!50, fill=white, minimum size=1.6mm},
		ajoint/.style={circle,inner sep=0pt, draw=blue!50, fill=white, minimum size=2.2mm},
		sjoint/.style={circle,inner sep=0pt, fill=black, minimum size=1mm}]

\def\jointbox{ +(0.18,-0.08) arc (0:180:0.18) -- cycle} 
\def\jointboxdown{ +(-0.18,0.08) arc (180:360:0.18) -- cycle}  
\def\sc{1.6}
\def\ca{(0.55635083*\sc,0.830947501*\sc)}
\def\cb{(-0.4114378*\sc,0.9114378*\sc)}
\def\cc{(0.99752066*\sc,0.070374129410*\sc)}
\def\wp{30}
\def\wpp{90}

\def\cp{(1.5*\sc,0.5*\sc)}

\def\trfix{ ++(-60:0.18) -- ++(-0.18,0) -- cycle}     
\def\ground{ ++(-120:0.18) ++(0.01,0) -- +(-120:0.1) ++(0.09,0) -- +(-120:0.1) ++(0.08,0) -- +(-120:0.1)}

\draw[->] (0,0) -- (4,0) node[right,fill=white] {\small $x$};
\draw[->] (0,0) -- (0,2) node[above,fill=white] {\small $y$};

\draw[thick, black] (0*\sc,0*\sc) node[right, xshift=7pt, yshift=7pt] {\scalebox{.91}{$\theta_1$}} -- \ca node[left, xshift=-9pt, yshift=2pt] {\scalebox{.91}{$\theta_2$}} coordinate(a) -- \cp coordinate (b) -- +(\wp:1*\sc) node[right, xshift=3pt, yshift=0pt]{\scalebox{.91}{$\theta_7,\, \theta_9$}} coordinate (c) -- +(\wpp:1*\sc) node[above right, xshift=3pt, yshift=-1pt]{\scalebox{.91}{$\theta_4$}} coordinate (d) -- (b);
\draw (0.7,0) arc (0:56:0.7);

\draw[thick] (1*\sc,0*\sc) coordinate (A3)-- +\cc node[right, xshift=3pt, yshift=2pt] {\scalebox{.91}{$\theta_8$}}coordinate(e) -- (c);
\draw[thick] (1*\sc,1*\sc) coordinate (A2)-- +\cb node[right, xshift=3pt, yshift=6pt] {\scalebox{.91}{$\theta_5$}}coordinate(f) -- (d);

\draw[line width=1.5pt] (0,0) -- (a) -- (b);
\draw[line width=1.5pt] (A2) -- (f) -- (d);
\draw[line width=1.5pt] (A3) -- (e) -- (c);

\draw[-] ($(a) + (0.26,0)$) arc (0:340:0.26);
\draw[dashed] (a) -- +(0.5,0);
\draw[-] ($(b) + (0.5,0)$) arc (0:30:0.5) node[right, xshift=1pt, yshift=-4pt] {\scalebox{.91}{$\theta_6$}};
\draw[dashed] (b) -- +(0.8,0);
\draw[-] ($(b) + (0.26,0)$) arc (0:90:0.26) node[right, xshift=1pt, yshift=3pt] {\scalebox{.91}{$\theta_3$}};



\node at (0,0) [ajoint] {};
\node at (a) [joint] {};
\node at (b) [joint] {};
\node at (c) [joint] {}; 
\node at (d) [joint] {}; 
\node at (e) [joint] {};
\node at (f) [joint] {};
\node at (A2) [ajoint] {}; 
\node at (A3) [ajoint] {};

\end{tikzpicture}\caption{a plane 3RRR-mechanism}\label{fig:3rrr}
\end{figure}
As in \cite{piippo:planar}, we set $\cos(\theta_i) = c_i$, $\sin(\theta_i) = s_i$, then $X$ is the real zero set of $I = \langle p_1, \ldots, p_{15} \rangle \leq  \R[\{c_i,s_i \mid i=1,\ldots,9\}]$, where
			\begin{align*}
			p_1 &= c_1 + c_2 + c_3 + c_4 + c_5  - 1;\\
			p_2 &= s_1 + s_2 + s_3 + s_4 + s_5  - 1;\\
			p_3 &= c_1 + c_2 + c_6 + c_7 + c_8  - 1;\\
			p_4 &= s_1 + s_2 + s_6 + s_7 + s_8;\\
			p_5 &= c_6 + c_9 - c_3\\
			p_6 &= s_6 + s_9 - s_3\\
			p_{6+i} &= c_i^2 + s_i^2 - 1, \quad i=1,\ldots,9.
			\end{align*}

We can check with \texttt{Singular}, that $\dim I = 3$, but the ideal $J$ generated by $I$ and the $15$-minors of the jacobian of $(p_1, \ldots, p_{15})$ has dimension $0$, which can be confirmed by analyzing the dimensions of $I + J_k$, where $J_1, \ldots, J_s$ are all the ideals given by the factorizing Gr\"obner Base algorithm (\texttt{facstd} in \texttt{Singular}). We also know, that the coordinate ring $A = \R[\{c_i,s_i\}]/I$ is a complete intersection ring, since $I$ is generated by $15$-elements, but then $A$ is Cohen Macaulay and we conclude:
\begin{itemize}
	\item[(1)] $I$ is equidimensional and radical, \cite[Cor. 18.14, Theorem~18.15]{eisenbud:comm_alg}.
	\item[(2)] The singular locus of $X$ is zero-dimensional, which follows from (1) and the general Jacobian criterion~\cite[Thm. 5.7.1]{greuel:singular_commutative}. 
	\item[(3)] $A$ is a normal ring, \cite[Theorem~18.15]{eisenbud:comm_alg},
	\item[(4)] All the components of $X$ are disjoint, since the singular intersection of components would have codimension $\geq 1$ according to Hartshorne's Connectedness Theorem~\cite[Thm. 18.13]{eisenbud:comm_alg}.
\end{itemize}  
Now at any point, the local ring is the local ring of an irreducible, normal, affine $\R$-variety, because of (4) and (3), and we can apply Efroymson's Criterion. We see at any point $p$ of $X$, that the extension of $I$ to the power series ring is real, if and only if $p$ is not isolated in the set of nonsingular real points of $X$. But the singular locus is of dimension $0$ (2) and we only need to check that a singularity is not isolated in $X$, to prove that it is not a manifold point. Since $X$ is given as the configuration space of a linkage we can often achieve this by geometric arguments. For example, in the following singular configuration of the mechanism, one of the legs can rotate freely, so $X$ is not a manifold there: 
\vspace{0.1cm}

\begin{figure}[h]
\begin{minipage}{0.5\textwidth}
\begin{center}
\begin{tikzpicture}
			[scale=1.1,joint/.style={circle,inner sep=0pt, draw=blue!50,fill=white, minimum size=1.3mm},
			ajoint/.style={circle,inner sep=0pt, draw=blue!50,fill=white, minimum size=2.2mm},
			sjoint/.style={circle,inner sep=0pt, fill=black, minimum size=1mm}]
			\def\r32{0.8660254}
			\def\sc{1.6}
			\def\jointbox{ +(0.18,-0.08) arc (0:180:0.18) -- cycle} 
			\def\jointboxdown{ +(-0.18,0.08) arc (180:360:0.18) -- cycle}  
			
			\def\trfix{ ++(-60:0.18) -- ++(-0.18,0) -- cycle}     
			\def\ground{ ++(-120:0.18) ++(0.01,0) -- +(-120:0.1) ++(0.09,0) -- +(-120:0.1) ++(0.08,0) -- +(-120:0.1)}

			\draw[->] (0,0) -- (4,0) node[right,fill=white] {\scriptsize $x$};
			\draw[->] (0,0) -- (0,2) node[above,fill=white] {\scriptsize $y$};

			\draw[thick, black] (0*\sc,0*\sc) -- ++(\r32*\sc,-0.5*\sc) coordinate (a) -- ++(1*\sc,0*\sc) coordinate (b) -- ++ (-\r32*\sc,0.5*\sc) coordinate (c) -- ++(0*\sc,-1*\sc) coordinate (d) -- ++(\r32*\sc,0.5*\sc);
			
			\draw[line width=1.2pt, black] (1*\sc,1*\sc) coordinate (e) -- (1*\sc,0) -- (1*\sc,-1*\sc);
			\draw[line width=1.2pt, black] (1*\sc,0*\sc) -- (2*\sc,0) coordinate (f) ;
			\draw[line width=1.2pt, black] (0*\sc,0*\sc) -- (a) -- (b);
			
			
			

			\node at (0,0) [ajoint] {};
			\node at (a) [joint] {};
			\node at (b) [joint] {};
			\node at (c) [ajoint] {}; 
			\node at (d) [joint] {}; 
			\node at (e) [ajoint] {};
			\node at (f) [joint] {};
			
			\end{tikzpicture}
		\end{center}	
\end{minipage}%
\begin{minipage}{0.5\textwidth}
\begin{tabular}{@{}cc} \toprule
			Variable & Value \\  \midrule
			$(c_1,s_1)$ 	& $\left(\frac{\sqrt{3}}{2}, -\frac{1}{2}\right)$		\\
			$(c_2,s_2)$ 	& $(1,0)$									\\
			$(c_3,s_3)$		& $\left(-\frac{3}{2}, -\frac{1}{2}\right)$		  	\\
			$(c_4,s_4)$		& $(0,1)$		\\
			$(c_5,s_5)$		& $(0,1)$	 	\\
			$(c_6,s_6)$	    & $\left(-\frac{3}{2}, \frac{1}{2} \right)$		 		\\
			$(c_7,s_7)$	    & $(0,1)$	 	\\
			$(c_8,s_8)$ 	& $(0,-1)$	\\
			$(c_9,s_9)$ 	& $(0,-1)$	\\ \bottomrule
		\end{tabular}
\end{minipage}\caption{A singular configuration of the 3RRR-mechanism}
\end{figure}